\documentclass[12pt]{amsart}
\usepackage[utf8]{inputenc}
\usepackage{amsmath,amssymb,amsthm,colonequals,mathrsfs,mathtools}
\usepackage{array,enumitem,yfonts}
\usepackage{comment,ulem}
\usepackage[alphabetic]{amsrefs}
\usepackage[all,cmtip]{xy}
\usepackage[colorlinks,anchorcolor=blue,citecolor=blue,linkcolor=blue,urlcolor=blue,bookmarksopen=true]{hyperref}
\allowdisplaybreaks

\usepackage[margin=1in]{geometry}
\usepackage{ulem}

\usepackage{tikz}
\usetikzlibrary{cd,arrows}
\tikzset{>=stealth}
\tikzcdset{arrow style=tikz}
\tikzset{link/.style={column sep=1.8cm,row sep=0.16cm}}

\AtBeginDocument{%
	\def\MR#1{}
}
\usepackage{fancyhdr}
\fancypagestyle{titlepage}
{
	\fancyhf{}

	\fancyfoot[l]{
	\href{https://mathscinet.ams.org/mathscinet/msc/msc2020.html}
	}
}

\newcommand{\bF}{\mathbb{F}}    

\newcommand{\bP}{\mathbb{P}}    

\newcommand{\Sing}{\operatorname{Sing}}

\newtheorem{thm}{Theorem}[section]
\newtheorem{prop}[thm]{Proposition}
\newtheorem{lemma}[thm]{Lemma}
\newtheorem{cor}[thm]{Corollary}
\numberwithin{equation}{section}

\theoremstyle{definition}

\newtheorem{defn}[thm]{Definition}
\newtheorem{eg}[thm]{Example}

\newtheorem{rmk}[thm]{Remark}

\begin{document}

\title{\bf Transverse linear subspaces to hypersurfaces over finite fields}
\author{Shamil Asgarli
	\and Lian Duan
	\and Kuan-Wen Lai}
\date{}

\newcommand{\ContactInfo}{{
\bigskip\footnotesize

\bigskip
\noindent S.~Asgarli,
\textsc{Department of Mathematics and Computer Science \\
Santa Clara University \\
Santa Clara, CA 95050, USA}\par\nopagebreak
\noindent\texttt{sasgarli@scu.edu}

\bigskip
\noindent L.~Duan,
\textsc{Institute of Mathematical Sciences \\
ShanghaiTech University \\
No.393 Middle Huaxia Road, Pudong New District, 
Shanghai, China}\par\nopagebreak
\noindent\texttt{duanlian@shanghaitech.edu.cn}

\bigskip
\noindent K.-W.~Lai,
\textsc{Institute of Mathematics, Academia Sinica \\
6F, Astronomy-Mathematics Building \\
No. 1, Sec. 4, Roosevelt Road, Taipei 106319, Taiwan}\par\nopagebreak
\noindent\textsc{Email:} \texttt{kwlai@gate.sinica.edu.tw}
}}

\maketitle
\thispagestyle{titlepage}

\begin{abstract} Ballico proved that a smooth projective variety $X$ of degree $d$ and dimension $m$ over a finite field of $q$ elements admits a smooth hyperplane section if $q\geq d(d-1)^{m}$. In this paper, we refine this criterion for higher codimensional linear sections on smooth hypersurfaces and for hyperplane sections on Frobenius classical hypersurfaces. We also prove a similar result for the existence of reduced hyperplane sections on reduced hypersurfaces.
\end{abstract}

\section{Introduction}
\label{sect:intro}

A classical theorem of Bertini asserts that a smooth projective variety $X\subset\bP^n$ defined over an infinite field $k$ admits a smooth hyperplane section defined over $k$. By applying this theorem repeatedly, one can obtain a linear section on $X$ of any dimension without extending the ground field $k$.

If $k=\bF_q$ is a finite field, then Bertini's theorem is no longer true in its original form as there are only finitely many hyperplanes in $\bP^n$ defined over $\bF_q$, and they could all happen to be tangent to $X$. As a concrete counterexample, see \cite{Kat99}*{Question 10}, \cite{Poo04}*{Theorem 3.1} or \cite{Asg19}*{Example 2.2}. There are two approaches to remedy this situation:
\begin{enumerate}[label=(\arabic*)]
    \item\label{Poonen-view}
    Instead of intersecting $X$ with hyperplanes, one could allow intersection with hypersurfaces of arbitrary degrees. This approach was taken by Poonen in \cite{Poo04}, where he proved the existence of a hypersurface $Y$ over the ground field such that the intersection $X\cap Y$ is smooth.
    
    \item\label{Ballico-view}
    Bertini's theorem is still valid if the cardinality of $\bF_q$ is sufficiently large with respect to $d\colonequals\deg(X)$. In this direction, Ballico \cite{Bal03} proved that if
    $$
        q\geq d(d-1)^{\dim X},
    $$
    then there exists an hyperplane $H$ over $\bF_q$ such that $X\cap H$ is smooth. Applying this result repeatedly, one can obtain a smooth linear section on $X$ over $\bF_q$ of any dimension.
\end{enumerate}

In the direction of \ref{Ballico-view}, Cafure--Matera--Privitelli \cite{CMP15}*{Corollary~6.6} and Matera--P\'erez--Privitelli \cite{MPP16}*{Theorem~3.6} extended Ballico's result to higher codimensional linear sections on possibly singular complete intersections. In the case of smooth hypersurfaces $X\subset\bP^n$ of degree $d$ over $\bF_q$, both results assert the existence of an $r$-dimensional linear subspace $L\subset\bP^n$ over $\bF_q$ such that $X\cap L$ is smooth provided that
\begin{equation}
\label{eqn:CMP-MPP-bound}
    q \geq d(d-1)^{r} + f_r(d)
\end{equation}
where $f_r(d)$ is a polynomial in $d$ of degree $r$ with coefficients depending on $r$.

In this paper, we first establish a statement similar to \cites{CMP15,MPP16} for smooth hypersurfaces via an independent approach. This new method allows us to construct inductively a flag of linear subspaces that satisfy a stronger notion of transversality. Moreover, our lower bound for $q$ is the same as \eqref{eqn:CMP-MPP-bound} except that $f_r(d)$ is replaced by a constant $\leq d$. In the following, we call an $r$-dimensional linear subspace in $\bP^n$ briefly as an \emph{$r$-plane}.

\begin{thm}
\label{mainthm:tranSp}
Let $X\subset \bP^n$ be a smooth hypersurface of degree $d$ defined over $\bF_q$ and pick any $0\leq r\leq n-1$. Suppose that
$$
    q \geq d(d-1)^{r}+\beta_r
    \quad\text{where}\quad
    \beta_r = \begin{cases}
        1\quad\text{if}\quad r\leq n-3, \\
        d\quad\text{if}\quad r=n-2, \\
        0\quad\text{if}\quad r=n-1.
    \end{cases}
$$
Then there exists a sequence of linear subspaces
$
    H_0\subset H_1\subset \dots \subset H_r
$
where each $H_i$ is an $i$-plane over $\bF_q$ that is \textbf{very transverse} to $X$ in the following sense:
\begin{itemize}
    \item $H_i$ is transverse to $X$, that is, $X\cap H_i$ is smooth, and
    \item $H_i$ is contained in a hyperplane over $\overline{\bF_q}$ that is transverse to $X$.
\end{itemize}
\end{thm}

A transverse hyperplane is automatically very transverse, so Theorem~\ref{mainthm:tranSp} recovers Ballico's result when $r = n-1$. For higher codimensions, the notion of very transversality becomes different from the usual transversality. As a simple example, consider the conic $C\subset\bP^2$ over a field of characteristic $2$ defined by the equation
\begin{equation}
\label{eqn:strange-conic}
    x^2 = yz.
\end{equation}
This is an example of a \emph{strange curve} as the point $[x:y:z] = [1:0:0]$ lies on every tangent line of $C$. Notice that this point is not on $C$, so it represents a transverse $0$-plane which is not very transverse. More discussions on very transversality and the proof of Theorem~\ref{mainthm:tranSp} will be given in Section~\ref{sect:very-tran-subsp}.

\begin{rmk}
For $d\geq 3$ and $r\geq 1$, we can improve $\beta_r=1$ to $\beta_r=0$ for $r\leq n-3$. Indeed, the quantity $d(d-1)^{r}$ is never a prime power in this case, and hence the hypothesis $q\geq d(d-1)^{r}+1$ can be relaxed to $q\geq d(d-1)^{r}$.
\end{rmk}

\begin{rmk}
When $n=3$ and $r=1$, we have $\beta_r=d$ and Theorem~\ref{mainthm:tranSp} implies that a smooth surface in $\bP^3$ admits a transverse $\bF_q$-line if $q\geq d(d-1)+d=d^2$. This result was proved using a different idea in our previous paper \cite{ADL21}*{Theorem~3.1}.
\end{rmk}

The second part of our paper focuses on a special type of hypersurfaces. Given a smooth hypersurface $X\subset\bP^n$ over $\bF_q$, we call $X$ \emph{Frobenius classical} if there exists a point $P\in X$ whose image under the $q$-th Frobenius endomorphism is outside the tangent hyperplane $T_PX$. Otherwise, we call $X$ \emph{Frobenius nonclassical}. Note that a hyperplane is Frobenius nonclassical by definition. Examples of Frobenius classical hypersurfaces include reflexive hypersurfaces \cite{ADL21}*{Theorem~4.5}. We expect such a hypersurface to have a transverse $r$-plane over $\bF_q$ provided that $q\geq\mathrm{O}(d^{\,r})$. As evidences, it is known that
\begin{itemize}
    \item a Frobenius classical curve $C\subset\bP^2$ of degree $d$ over $\bF_q$ admits a transverse $\bF_q$-line if $q\geq d-1$ \cite{Asg19Thesis}*{Theorem 3.3.1}.
    \item a Frobenius classical surface $S\subset\bP^3$ of degree $d$ over $\bF_q$ admits a transverse $\bF_q$-line when $q\geq cd$ for some constant $c>0$ \cite{ADL21}*{Theorem~0.1}.
\end{itemize}
In Section~\ref{sect:refine-Frob}, we prove this conjecture for hyperplane sections on hypersurfaces:

\begin{thm}
\label{thm:refine-Ballico}
Let $X\subset\bP^n$ be a smooth Frobenius classical hypersurface of degree $d$ defined over $\bF_q$. Suppose
$$
    q\geq c_d\cdot d(d-1)^{n-2}
    \qquad\text{where}\qquad
    c_d = \begin{cases}
    1 &\text{for}\quad d=2\\
    (3d+3)(3d-1)^{-1} &\text{for}\quad d\geq 3.
    \end{cases}
$$
Then there exists an $\bF_q$-hyperplane $H\subset\bP^n$ such that $X\cap H$ is smooth.
\end{thm}

Note that $c_d$ strictly decreases in $d$ for $d\geq 3$ starting from $c_3 = 3/2$, and $c_d\to 1$ as $d\to \infty$. In particular, the statement of Theorem~\ref{thm:refine-Ballico} still holds if $c_d$ is replaced by the constant $3/2$. Given a fixed ambient dimension $n$, this theorem improves the bound provided by Ballico's theorem by a factor of $d-1$.

In general, a Bertini type theorem concerns the existence of linear sections that inherit some nice properties such as smoothness, reducedness, irreducibility, and normality, from the ambient variety. Recall that a scheme $X$ is \emph{reduced} if $\mathcal{O}_X(U)$ contains no nonzero nilpotent element for every open subset $U\subset X$ \cite{Har77}*{II, Section~3}. Because $\bF_q$ is perfect, reducedness over the ground field $\bF_q$ and the algebraic closure $\overline{\bF_q}$ are equivalent. In particular, a hypersurface $X\subset\bP^n$ defined by a polynomial $F\in\bF_q[x_0,\dots,x
_n]$ is reduced provided that the quotient ring $\overline{\bF_q}[x_0, ..., x_n]/(F)$ contains no nonzero nilpotent element. Combining these ingredients, we obtain the criterion: \emph{a hypersurface $X = \{F = 0\}\subset\bP^n$ over $\bF_q$ is reduced if and only if in the factorization of $F$ into irreducible polynomials over $\overline{\bF_q}$
$$
    F = \prod_{i=1}^m G_i
    \;\in\;\overline{\bF_q}[x_0,\dots,x_n]
$$
all the factors $G_i$ are coprime to each other.}

Our third result concerns the existence of reduced hyperplane sections, and can be viewed as Bertini's theorem for reducedness over finite fields.

\begin{thm}
\label{main-theorem-red}
Let $X\subset \bP^n$ be a reduced hypersurface of degree $d$ over $\bF_q$. Then there exists a hyperplane $H\subset\bP^n$ over $\bF_q$ such that $X\cap H$ is reduced and has dimension $n-2$ if $q$ is greater than or equal to a constant depending only on $n$ and $d$ as given below:
\begin{center}
\setlength\extrarowheight{3pt}
\begin{tabular}{|c|c|c|c|}
    \hline
     & $n=2$ & $n=3$ & $n\geq 4$ \\[3pt]
    \hline
    $q\geq$ & $\frac{3}{2}d(d-1)$ & $d(d-1)+1$ & $d$ \\[3pt]
    \hline
\end{tabular}
\end{center}
\end{thm}

By applying this theorem repeatedly, one can deduce similar statements for higher codimensional linear sections; see Corollary~\ref{cor:reduced-r-plane}. Notice that, in the case $n=2$, the theorem asserts the existence of an $\bF_q$-line $H$ in $\bP^2$ meeting a reduced plane curve in its smooth locus transversely. Theorem~\ref{main-theorem-red} will be proved in Section~\ref{sect:reduced-hypersection}.

\subsection*{Acknowledgements}
We are very grateful to an anonymous referee for a detailed review and helpful suggestions that improved both the accuracy and clarity of our work. The first author was partially supported by NSF grant 1701659, and by a postdoctoral research fellowship from the University of British Columbia. The third author was supported by the ERC Synergy Grant HyperK (ID: 854361).

\section{Existence of very transverse linear subspaces}
\label{sect:very-tran-subsp}

Assume that one would like to prove Theorem~\ref{mainthm:tranSp} for merely transverse linear subspaces by induction on $r$. Then, given an $(r-1)$-plane $H_{r-1}\subset\bP^n$ over $\bF_q$ transverse to the smooth hypersurface $X$, one needs to find an $r$-plane $H_r\supset H_{r-1}$ also transverse to $X$. However, such an $H_r$ may not exist in general in view of the strange conic \eqref{eqn:strange-conic}. In order to remedy this situation, we run the induction process for transverse linear subspaces that satisfy additional properties:

\begin{defn}
Let $X\subset\bP^n$ be a smooth hypersurface over an arbitrary field $k$. We say a linear subspace $H\subset\bP^n$ is \emph{very transverse} to $X$ if
\begin{enumerate}[label=(\arabic*)]
    \item\label{very-trans:trans}
    it is transverse to $X$, and
    \item\label{very-trans:very}
    it is contained in a hyperplane over $\overline{k}$ that is transverse to $X$.
\end{enumerate}
\end{defn}

Property~\ref{very-trans:very} can be translated into the following equivalent form in terms of projective duality. Consider the Gauss map
$$\xymatrix{
    \gamma\colon X\ar[r] & (\bP^n)^{\ast}
    : P\ar@{|->}[r] & T_PX.
}$$
Let $X^{\ast}\colonequals\gamma(X)$ be the projective dual of $X$ and $H^{\ast}\subset(\bP^n)^{\ast}$ be the subspace consisting of the hyperplanes in $\bP^n$ that contain $H$. Then \ref{very-trans:very} is equivalent to
$$
    H^{\ast}\not\subset X^{\ast}.
$$
In characteristic zero, a transverse linear subspace is automatically very transverse due to the general version of Bertini's theorem (see \cite{Kle74}*{Corollary~5}). On the other hand, there exist odd dimensional smooth quadrics in characteristic~$2$ which admit linear subspaces that are transverse but not very transverse:

\begin{eg}
\label{eg:strange-quadric}
Suppose that $X\subset\bP^n$ is a smooth hypersurface of degree at least $2$ over an algebraically closed field. Then $X$ is \emph{strange}, meaning that its tangent hyperplanes contain a common point $P\in\bP^n$, if and only if it is an odd dimensional quadric in characteristic~$2$ \cite{KP91}*{Theorem~7}. In this case, one can express $X$ as
\begin{equation}
\label{eqn:strange-quadric}
    x_0^2 + \sum_{i=1}^m x_{2i-1}x_{2i} = 0
    \quad\text{where}\quad
    m = \frac{n}{2}
\end{equation}
and let $P$ be the point
$
    [x_0:x_1:\dots:x_n] = [1:0:\dots:0].
$
Since $P$ is not on $X$ and every hyperplane containing $P$ is tangent to $X$, it represents a $0$-plane that is transverse but not very transverse to $X$. More generally, for every even $0\leq r < n$, the $r$-plane
$$
    H\colonequals\{
        x_{r+1} = \dots = x_n = 0
    \}
    \cong\bP^r
$$
is transverse as it intersects $X$ in a strange quadric defined by \eqref{eqn:strange-quadric} with $m$ replaced by $r$. Moreover, the hyperplanes containing $H$ also contain $P$, hence they are tangent to $X$; consequently, $H$ is not very transverse.
\end{eg}

Let us first show that, in the cases $d = 1,2$, Theorem~\ref{mainthm:tranSp} follows easily from the main result of \cite{Bal03}.

\begin{prop}
\label{prop:linear-and-quadric}
Let $X\subset \bP^n$ be a hyperplane or a smooth quadric over $\bF_q$. Then there exists a sequence of linear subspaces
$
    H_0\subset H_1\subset\dots\subset H_{n-1}
$
where each $H_r$ is an $r$-plane defined over $\bF_q$ and very transverse to $X$.
\end{prop}

\begin{proof}
Ballico \cite{Bal03} proved that if
$
    q\geq d(d-1)^{\dim X}
$
then there is an $\bF_q$-hyperplane $H_{n-1}$ transverse to $X_{n-1}\colonequals X$. Note that $d(d-1)^{\dim X}$ equals $0$ when $d=1$ and equals $2$ when $d=2$, so the above inequality always holds in our situation. Therefore, we can take $X_{n-2}\colonequals X_{n-1}\cap H_{n-1}$, consider it as a hypersurface in $H_{n-1}\cong\bP^{n-1}$, and repeat the same process to find an $(n-2)$-plane $H_{n-2}\subset H_{n-1}$ over $\bF_q$ transverse to $X_{n-2}$. Thus, by induction, we have a sequence of transverse linear subspaces
$
    H_0\subset H_1\subset\cdots\subset H_{n-1}
$
such that $\dim H_r=r$. Notice that each $H_r$ in this sequence is very transverse since they are all contained in the transverse hyperplane $H_{n-1}$.
\end{proof}

\subsection{Strategy for proving Theorem~\ref{mainthm:tranSp}}
\label{subsect:very-transverse-dual}

Let $X\subset\bP^n$ be a smooth hypersurface over $\bF_q$ that admits a very transverse $(r-1)$-plane $H_{r-1}\subset\bP^n$. Among the $r$-planes over $\bF_q$ that contain $H_{r-1}$, we would like to estimate the number of bad choices, namely, the $r$-planes that are \emph{not} very transverse to $X$.

By definition, a linear subspace $H\subset\bP^n$ is not very transverse to $X$ if and only if
\begin{enumerate}[label=(\roman*)]
    \item\label{bad-plane:not-trans}
    it is not transverse to $X$, or
    \item\label{bad-plane:non-proper-dual}
    the dual subspace $H^\ast\subset(\bP^n)^\ast$ is contained in $X^\ast$.
\end{enumerate}
Our estimates for the numbers of bad $r$-planes of these two types are established using geometry of the dual hypersurface $X^\ast$. Let us deal with $r$-planes of type~\ref{bad-plane:non-proper-dual} first:

\begin{prop}
\label{prop:non-proper-dual}
Let $X\subset\bP^n$ be a smooth hypersurface of degree $d$ over a field $k$ which admits a very transverse $(r-1)$-plane $H_{r-1}$ over $k$. Then the number of $r$-planes $H$ that contain $H_{r-1}$ and satisfy $H^{\ast}\subset X^{\ast}$ is at most $d(d-1)^{n-1}$.
\end{prop}

\begin{proof}
By hypothesis, $H_{r-1}^{\ast}\cap X^{\ast}$ is a hypersurface in $H_{r-1}^{\ast}\cong\bP^{n-r}$ whose degree is equal to $\deg(X^{\ast}) \leq d(d-1)^{n-1}$ by Pl\"ucker's formula \cite{Kle86}*{Propositions 2 and 9}. For each $r$-plane $H$ containing $H_{r-1}$, its dual $H^\ast$ appears as a hyperplane in $H_{r-1}^\ast$, and $H^\ast\subset X^\ast$ means that $H^\ast$ appears as a linear component of the hypersurface $H_{r-1}^{\ast}\cap X^{\ast}$. As this hypersurface has degree at most $d(d-1)^{n-1}$, the statement follows.
\end{proof}

The estimate for the number of bad $r$-planes of type~\ref{bad-plane:not-trans} is more complicated. We will turn this into a point counting problem on a certain projective scheme, and leave the explicit computation to the next subsection.

In the following, we fix a smooth hypersurface $X\subset\bP^n$ of degree $d$ over a field $k$ and assume that it admits a very transverse $(r-1)$-plane $H_{r-1}$ over $k$. Let
$
    \gamma\colon X\longrightarrow(\bP^n)^\ast
$
be the Gauss map associated with $X$ and define
\begin{equation}
\label{eqn:gauss-preimage}
    Y_{r-1}\colonequals\gamma^{-1}(H_{r-1}^\ast).
\end{equation}

\begin{prop}
\label{prop:gauss-preimage}
The scheme $Y_{r-1}$ defined above, as a subscheme of $\bP^n$, is a complete intersection of codimension $r+1$ and multidegree $(d,d-1,\dots,d-1)$.
\end{prop}

\begin{proof}
By hypothesis, the intersection $X^\ast\cap H_{r-1}^\ast$ is a hypersurface in $H_{r-1}^\ast\cong\bP^{n-r}$, so it has dimension $n-r-1$. Since $Y_{r-1} = \gamma^{-1}(X^\ast\cap H_{r-1}^\ast)$, the finiteness of Gauss map \cite{Zak93}*{Corollary~I.2.8} implies that $Y_{r-1}$ has dimension $n-r-1$ as well. Thus it has codimension $r+1$ in $\bP^n$. Let $L_i = 0$, $i=1,\dots,r$, be the linear equations that cut out $H_{r-1}^\ast$ in $(\bP^n)^\ast$. By construction, $Y_{r-1}$ is the zero locus on $X$ of the polynomials $\gamma^\ast L_i$ for $i=1,\dots,r$. This shows that $Y_{r-1}$ is a complete intersection, and its multidegree is determined by $\deg(X) = d$ and $\deg(\gamma^\ast L_i) = d-1$ for all $i$.
\end{proof}

Let $\pi\colon\bP^n\dashrightarrow\bP^{n-r}$ be the projection from $H_{r-1}$. Note that taking a point $P\in\bP^{n-r}$ to its preimage $\pi^{-1}(P)\subset\bP^n$ defines a bijection between the following sets
\begin{equation}
\label{eqn:proj-image}
\xymatrix{
    \bP^{n-r}(k')\ar[r]^-\sim &
    \{
        r\text{-planes }H\subset\bP^n\text{ over }k'\text{ such that }H\supset H_{r-1}
    \}
}
\end{equation}
where $k'$ is any field extension over $k$. In the following, the \emph{proper image} $\pi(Y_{r-1})$ means the closure of $\pi(Y_r^\circ)$, where $Y_{r-1}^\circ\subset Y_{r-1}$ be the subset where $\pi$ is well-defined.

\begin{lemma}
\label{lem:proj-image}
Let $P\in\bP^{n-r}$ be a $k$-point outside the proper image $\pi(Y_{r-1})\subset\bP^{n-r}$ of the scheme defined in \eqref{eqn:gauss-preimage}. Then the image $r$-plane $H$ of $P$ under bijection~\eqref{eqn:proj-image} is defined over $k$ and transverse to $X$.
\end{lemma}

\begin{proof}
Assume, to the contrary, that $H$ is not transverse to $X$, namely, there exists a point $Q\in X\cap \pi^{-1}(P)$ such that $Q\in H\subset T_QX$. The inclusion relation $H_{r-1}\subset H$ implies $H_{r-1}\subset T_QX$, or equivalently, $\gamma(Q)\in H_{r-1}^\ast$, which implies $Q\in Y_{r-1}$. We claim that $Q\notin H_{r-1}$. Indeed, if $Q\in H_{r-1}$, then the same relation $H_{r-1}\subset H$ implies $Q\in H_{r-1}\subset T_QX$, whence $H_{r-1}$ is not transverse to $X$, contradiction. We conclude that $Q\in Y_{r-1}\setminus H_{r-1}$, so the projection $\pi$ is well defined at $Q$, and $P=\pi(Q)\in\pi(Y_{r-1})$, contradiction.
\end{proof}

Suppose that the ground field $k$ is finite. As a consequence of Lemma~\ref{lem:proj-image}, to find an $r$-plane $H\supset H_{r-1}$ over $k$ which is transverse to $X$, it is sufficient to show that the $k$-points on $\bP^{n-r}$ are strictly more than the $k$-points on $\pi(Y_{r-1})$.

\subsection{Proof of Theorem~\ref{mainthm:tranSp}}
\label{subsect:degree>=3}

In the following, we will compute an estimate for the number of bad $r$-planes of type~\ref{bad-plane:not-trans} and then conclude the proof of Theorem~\ref{mainthm:tranSp}. Some technical lemmas needed in the process will be postponed to the next subsection.

\begin{defn}
For every integer $m\geq 0$ and $q$ a power of a prime number, we define
$$
    \theta_m(q)\colonequals
    \#\bP^m(\bF_q)
    = \frac{q^{m+1}-1}{q-1}
$$
and set $\theta_m(q) = 0$ for $m<0$. We will write $\theta_m(q)$ briefly as $\theta_m$ if there is no ambiguity about $q$ from the context.
\end{defn}

\begin{lemma}
\label{lem:couvreur-improved}
Let $X\subset\bP^n$ be a reduced subscheme over $\bF_q$ and write
$
    X = \bigcup_{i=1}^t X_i
$
where $X_i$ is an irreducible component of dimension $m_i<n$ and degree $d_i$. Denote $d\colonequals\sum_{i=1}^{t} d_i$ and $m\colonequals\max(m_1, ..., m_t)$. Then
$$
    \#X(\bF_q)
    \leq d(\theta_m - \theta_{2m-n}) + \theta_{2m-n}.
$$
\end{lemma}

\begin{proof}
We use the inequality from \cite{Cou16}*{Theorem 3.1}, which states
$$
    \# X(\bF_q)
    \leq\sum_{i=1}^{t} d_i(\theta_{m_i} - \theta_{2m_i-n})+ \theta_{2m-n}.
$$
Based on this, it is sufficient to verify that
\begin{align}
\label{eq:couvreur}
    \theta_{m_i}-\theta_{2m_i-n}\leq \theta_{m}-\theta_{2m-n}
    \qquad\text{for each}\qquad
    i\in\{1, \dots, t\}
\end{align}
Let us proceed by three cases:

\begin{description}[leftmargin=0cm]
\item[Case~($2m-n<0$)] Then $2m_i-n<0$ and \eqref{eq:couvreur} reduces to $\theta_{m_i}\leq \theta_{m}$ which clearly holds.

\smallskip
\item[Case~($2m-n\geq 0$ and $2m_i-n\geq 0$)]
Rearrange \eqref{eq:couvreur} as $\theta_{2m-n}-\theta_{2m_{i}-n}\leq \theta_{m}-\theta_{m_i}$, which is the same as
$$
    \frac{q^{2m}-q^{2m_i}}{q^n} \leq q^{m}-q^{m_i},
    \qquad\text{or equivalently,}\qquad
    q^{m} + q^{m_i}\leq q^n.
$$
The last inequality holds since $q^{m}+q^{m_i}\leq q^{n-1}+q^{n-1}=2q^{n-1}\leq q^{n}$.

\smallskip
\item[Case~($2m-n\geq 0$ and $2m_i-n < 0$)]
Then \eqref{eq:couvreur} reduces to $\theta_{m_i}+\theta_{2m-n}\leq \theta_{m}$, that is,
$$
    q^{m_i+1} + q^{2m-n+1} \leq q^{m+1}+1.
$$
The hypothesis implies $2m-n > 2m_i-n$ and thus $m\geq m_i+1$. We also have $n\geq m+1$, hence $m\geq 2m-n+1$. Therefore,
$$
    q^{m+1}+1
    \geq q^{m}+q^{m}+1
    \geq q^{m_i+1} + q^{2m-n+1} + 1
    > q^{m_i+1} + q^{2m-n+1}
$$
as desired.
\qedhere
\end{description}
\end{proof}

\begin{lemma}
\label{lem:inductive-step}
Let $X\subset\bP^n$ be a smooth hypersurface of degree $d\geq 2$ over $\bF_q$. Assume that $1\leq r\leq n-1$ and that $X$ admits a very transverse $(r-1)$-plane $H_{r-1}\subset\bP^n$ over $\bF_q$. Then $X$ admits a very transverse $r$-plane $H_{r}\supset H_{r-1}$ over $\bF_q$ provided that
$$
    q \geq d(d-1)^{r}+\beta_r
    \qquad\text{where}\qquad
    \beta_r = \begin{cases}
        1\quad\text{if}\quad r\leq n-3, \\
        d\quad\text{if}\quad r=n-2, \\
        0\quad\text{if}\quad r=n-1.
    \end{cases}
$$
\end{lemma}

\begin{proof}
Let $\gamma\colon X\longrightarrow(\bP^n)^{\ast}$ be the Gauss map associated with $X$. By Proposition~\ref{prop:gauss-preimage}, the preimage
$
    Y_{r-1}\colonequals\gamma^{-1}(H_{r-1}^\ast)
$
is a complete intersection in $\bP^n$ of dimension $n-r-1$ and of degree $d(d-1)^r$. Let $\pi\colon\bP^n\dashrightarrow\bP^{n-r}$ be the projection from $H_{r-1}$ and consider the proper image $\pi(Y_{r-1})\subset\bP^{n-r}$. If we write
$
    \pi(Y_{r-1}) = \bigcup_{i=1}^tY'_i,
$
where $Y'_i$ is an irreducible component of dimension $m_i$ and degree $d_i$, then
$$
    m\colonequals\max(m_1,\dots,m_t)\leq n-r-1
    \qquad\text{and}\qquad
    \sum_{i=1}^td_i\leq d(d-1)^r.
$$
It follows from Lemma~\ref{lem:couvreur-improved} that
$$
    \#\pi(Y_{r-1})(\bF_q)
    \leq d(d-1)^{r}\left(
        \theta_{m} - \theta_{2m-n}
    \right) + \theta_{2m-n}.
$$

According to Lemma~\ref{lem:proj-image}, the $\bF_q$-points in $\bP^{n-r}$ outside $\pi(Y_{r-1})$ one-to-one correspond to the $r$-planes containing $H_{r-1}$ over $\bF_q$ that are transverse to $X$. On the other hand, there are at most $d(d-1)^{n-1}$ many $r$-planes $H_r\supset H_{r-1}$ that satisfy $H_r^\ast\subset X^\ast$ by Proposition~\ref{prop:non-proper-dual}. Also recall that the transverse hyperplanes are automatically very transverse. As a consequence, there exists $H_r$ as in the statement provided that
\begin{equation}
\label{eqn:key-inequality}
    \theta_{n-r} > 
    d(d-1)^{r}\left(
        \theta_{m} - \theta_{2m-n}
    \right) + \theta_{2m-n} + \delta_r
\end{equation}
where $\delta_r = d(d-1)^{n-1}$ if $r\leq n-2$ and $\delta_r = 0$ if $r = n-1$. The proof of \eqref{eqn:key-inequality} is purely numerical and will be established via Lemmas~\ref{lem:delta-beta} and \ref{lem:key-inequality}.
\end{proof}

\begin{proof}[Proof of Theorem~\ref{mainthm:tranSp}]
The cases $d=1,2$ are already proved in Proposition~\ref{prop:linear-and-quadric}, so we assume $d\geq 3$ in the following. To prove the theorem, we will establish the existence of linear subspaces $H_0\subset\cdots\subset H_s$, where each $H_i$ is a very transverse $i$-plane over $\bF_q$, for all $0\leq s\leq r$ by induction on $s$. 

In the initial case $s = 0$, we need to find a point $P\in\bP^n(\bF_q)\setminus X(\bF_q)$. By the Homma--Kim bound \cite{HK13_bound}*{Theorem~1.2},
$$
    \#X(\bF_q)\leq
    (d-1)q^{n-1} + dq^{n-2} + \theta_{n-3}.
$$
It is sufficient to prove that
$$
    (d-1)q^{n-1} + dq^{n-2} + \theta_{n-3}
    < \theta_{n}.
$$
A straightforward computation reduces the last inequality to $d-1 < q$, which follows from our hypothesis since
$$
    q\geq d(d-1)^r+\beta_r \geq d > d-1.
$$
Hence there exists a point $P\in\bP^n(\bF_q)\setminus X(\bF_q)$. Note that $P$ is non-strange, namely, is contained in a transverse hyperplane over $\overline{\bF_q}$ because of $d\geq 3$. Therefore, $P$ is a $0$-dimensional linear subspace very transverse to $X$.

Before entering the inductive step, let us prove that
\begin{equation}
\label{eqn:ineq_induction}
    d(d-1)^\ell + \beta_\ell
    \;\geq\;
    d(d-1)^{\ell-1}+\beta_{\ell-1}
    \quad\text{for all}\quad
    1\leq\ell\leq r.
\end{equation}
This implication holds for all $\ell\leq n-3$ because $\beta_\ell = \beta_{\ell-1} = 1$ in these cases. It also holds for $\ell = n-2$ because $\beta_{n-2} = d > 1 = \beta_{n-3}$. When $\ell = n-1$, the desired inequality
$$
    d(d-1)^{n-1}
    \;\geq\;
    d(d-1)^{n-2}+d
$$
can be derived easily from the assumption $d\geq 3$. As a consequence, combining \eqref{eqn:ineq_induction} with the hypothesis $q\geq d(d-1)^r + \beta_r$ gives us
\begin{equation}
\label{eqn:ineq_hypo-s}
    q\geq d(d-1)^s + \beta_s
    \quad\text{for all}\quad
    0\leq s\leq r.
\end{equation}

Now assume that there exists a sequence of linear subspaces $H_0\subset\cdots\subset H_{s-1}$ for some $s\in\{1,\dots,r\}$ where each $H_i$ is a very transverse $i$-plane over $\bF_q$. Then Lemma~\ref{lem:inductive-step}, together with \eqref{eqn:ineq_hypo-s}, implies that there exists a very transverse $s$-plane $H_s\supset H_{s-1}$ over $\bF_q$. This establishes the inductive step and thus finishes the proof.
\end{proof}

\subsection{Some numerical lemmas}
\label{subsect:numerical_lemma}

Here we establish inequality~\eqref{eqn:key-inequality} which is needed in the proof of Lemma~\ref{lem:inductive-step}.

\begin{lemma}
\label{lem:delta-beta}
Let $n$, $r$, $d$ be positive integers that satisfy $1\leq r\leq n-1$ and $d\geq 2$. Define
$$
    \delta_r\colonequals\begin{cases} 
        d(d-1)^{n-1} &\text{if}\quad r\leq n-2 \\ 
        0 &\text{if}\quad r=n-1 \\ 
    \end{cases}
    \qquad\text{and}\qquad
    \beta_r\colonequals\begin{cases}
        1\quad\text{if}\quad r\leq n-3 \\
        d\quad\text{if}\quad r=n-2 \\
        0\quad\text{if}\quad r=n-1
    \end{cases}
$$
Then, for every integer $q$, the inequality
$
    q \geq d(d-1)^{r}+\beta_r
$
implies
$
    (q-1)q^{-n+r}\delta_r\leq\beta_r.
$
\end{lemma}

\begin{proof}
The implication is obvious when  $r=n-1$ since $\delta_r = \beta_r = 0$ in this case. Assume $r\leq n-2$. Using $q\geq d(d-1)^r+\beta_r$, we obtain
\begin{equation}
\label{eqn:delta-beta}
\begin{aligned}
    (q-1)q^{-n+r}\delta_r
    &< q\cdot q^{-n+r}\delta_r
    = \frac{\delta_r}{q^{n-r-1}}
    \leq\frac{\delta_r}{(d(d-1)^r + \beta_r)^{n-r-1}}\\
    &\leq\frac{\delta_r}{(d(d-1)^r)^{n-r-1}}
    = \frac{d(d-1)^{n-1}}{(d(d-1)^r)^{n-r-1}}.
\end{aligned}
\end{equation}
When $r = n-2$, the last term equals $d-1< d = \beta_r$, so the statement holds. Assume $r\leq n-3$. Together with $r\geq 1$, we obtain $n\geq r+3\geq r + 2 + 1/r$. It follows that
\begin{equation}
\label{eqn:1<=r<=n-3}
    n-r-1\geq 1+\frac{1}{r},
    \qquad\text{or equivalently,}\qquad
    r(n-r-1)\geq r+1.
\end{equation}
The last term of \eqref{eqn:delta-beta} can be rewritten as
$$
    \frac{d(d-1)^{n-1}}{d\cdot d^{n-r-2} (d-1)^{r(n-r-1)}}
    <\frac{d(d-1)^{n-1}}{ d\cdot (d-1)^{n-r-2} (d-1)^{r(n-r-1)}}
    = \frac{(d-1)^{r+1}}{(d-1)^{r(n-r-1)}}.
$$
Then \eqref{eqn:1<=r<=n-3} implies that the last term is at most $1 = \beta_r$. This establishes the statement.
\end{proof}

\begin{lemma}
\label{lem:key-inequality}
Retain the hypothesis from Lemma~\ref{lem:delta-beta} and assume that $q$ is a power of a prime number. Let $m$ be an integer that satisfies $0\leq m\leq n-r-1$. Then
$$
    \theta_{n-r}(q) > 
    d(d-1)^{r}\left(
        \theta_{m}(q) - \theta_{2m-n}(q)
    \right) + \theta_{2m-n}(q) + \delta_r.
$$
\end{lemma}

\begin{proof}
We proceed by two cases:

\begin{description}[leftmargin=0cm]

\item[Case~($2m-n\geq 0$)]
In this case, the desired inequality is
$$
    \frac{q^{n-r+1}-1}{q-1}
    > d(d-1)^{r}\left(
        \frac{q^{m+1}-1}{q-1}
        - \frac{q^{2m-n+1}-1}{q-1}
    \right)
    + \frac{q^{2m-n+1}-1}{q-1} + \delta_r.
$$
Multiplying both sides by $q^{-n+r}(q-1)$ and rearranging, we obtain
\begin{equation}
\label{eqn:equivalent-form}
    q > d(d-1)^{r}\left(
        q^{m-n+r+1} - q^{2m-2n+r+1}
    \right) + q^{2m-2n+r+1} + (q-1)q^{-n+r}\delta_r.
\end{equation}
Consider the right hand side as a function $g(m)$ in $m$. Taking derivative gives
$$
    g'(m)=d(d-1)^{r}\left(
        q^{m-n+r+1}\ln(q) - 2q^{2m-2n+r+1}\ln(q)
    \right) + 2q^{2m-2n+r+1}\ln(q).
$$
The assumption $m\leq n-r-1$ implies $n\geq m+r+1 > m+1$. It follows that
$$
    m-n+r = m+(n-2n)+r > 2m-2n+r+1.
$$
Hence $q^{m-n+r+1}\geq 2q^{m-n+r} > 2q^{2m-2n+r+1}$, which implies $g'(m)>0$, so $g(m)$ is increasing in $m$. As a consequence, it is sufficient to prove \eqref{eqn:equivalent-form} when $m$ attains the maximal possible value $n-r-1$, that is,
$$
    q > d(d-1)^{r}\left(
        1-q^{-r-1}
    \right) + q^{-r-1} + (q-1)q^{-n+r}\delta_r.
$$
We establish this via the following inequalities:
\begin{align*}
    q \geq d(d-1)^{r} + \beta_r
    &> d(d-1)^{r} + q^{-r-1}\underbrace{\left(1-d(d-1)^{r}\right)}_{<0} + \beta_r \\
    &\geq d(d-1)^{r} - q^{-r-1}d(d-1)^{r}+q^{-r-1} + (q-1) q^{-n+r} \delta_r
    \tag{by Lemma~\ref{lem:delta-beta}}\\
    &= d(d-1)^{r}\left(1-q^{-r-1}\right)+q^{-r-1} + (q-1)q^{-n+r}\delta_r.
\end{align*}

\medskip
\item[Case~($2m-n < 0$)]
In this case, the desired inequality reduces to:
$$
    \frac{q^{n-r+1}-1}{q-1}
    > d(d-1)^{r}\cdot\left(
        \frac{q^{m+1}-1}{q-1}
    \right) + \delta_r.
$$
Multiplying both sides by $q^{-n+r}(q-1)$ and rearranging the terms gives
$$
    q > d(d-1)^{r}q^{m+1-n+r}
    + (1 - d(d-1)^{r})q^{-n+r}
    + (q-1)q^{-n+r}\delta_r.
$$
Notice that $m\leq n-r-1$ implies $m+1-n+r\leq 0$. Then the above inequality follows from
\begin{align*}
    q\geq d(d-1)^{r} + \beta_r
    &\geq  d(d-1)^{r}q^{m+1-n+r} + \beta_r \\
    &> d(d-1)^{r}q^{m+1-n+r} + \underbrace{(1 - d(d-1)^r)}_{<0}q^{-n+r} + (q-1)q^{-n+r} \delta_r.
    \tag{by Lemma~\ref{lem:delta-beta}}
\end{align*}
\end{description}
This completes the proof.
\end{proof}

\section{Smooth sections on Frobenius classical hypersurfaces}
\label{sect:refine-Frob}

We prove Theorem~\ref{thm:refine-Ballico} in this section. In order to prove this result, we need to estimate the number of hyperplanes over the ground field which are tangent to a smooth hypersurface $X$. Our main strategy is to turn counting such hyperplanes into counting points on a certain $0$-dimensional subscheme of $X$.

\subsection{Tangent hyperplanes over the ground field}
\label{subsect:ground_tangent}

Let $X\subset\bP^n$ be a hypersurface defined over $\overline{\bF_q}$, where $q$ is a fixed prime power, and let
$
    F = F(x_0,\dots,x_n)
$
be the defining polynomial of $X$. Consider the $2\times(n+1)$ matrix
$$
    M\colonequals
    \begin{pmatrix}
        F_0 & \cdots & F_n \\
        F_0^q & \cdots & F_n^q
    \end{pmatrix}
    \qquad\text{where}\qquad
    F_i\colonequals\frac{\partial F}{\partial x_i}.
$$
For each $(i,j)$ such that $0\leq i<j\leq n$, the maximal minor given by the $i$-th and $j$-th columns of this matrix determines a hypersurface
$$
    D_{ij}\colonequals\{
        F_iF_j^q - F_i^qF_j = 0
    \}\subset\bP^n
$$
of degree $(d-1)(q+1)$ over $\overline{\mathbb{F}_q}$. Let us define
$$
    Z_X\colonequals X\cap\bigcap_{0\leq i<j\leq n} D_{ij}.
$$

\begin{prop}
\label{prop:singularOrGroundTangent}
Let $X\subset\bP^n$ be a hypersurface over $\overline{\bF_q}$. Then a point $P\in X(\overline{\bF_q})$ belongs to $Z_X$ if and only if $P$ is a singular point of $X$ or $T_PX$ is defined over $\bF_q$.
\end{prop}

\begin{proof}
Observe that $P\in Z_X$ if and only if the matrix $M$ has rank equal to $0$ or $1$ when evaluated at $P$, which happens if and only if
$$
    (F_0(P),\dots, F_n(P)) = (0,\dots,0)
    \quad\text{or}\quad
    [F_0(P):\dots: F_n(P)] \in \bP^n(\bF_q),
$$
and thus correspond to the two conditions in the statement.
\end{proof}

\begin{rmk}
Here is a more geometric way to view the locus $Z_X\subset X$. Let $\{y_0,\dots,y_n\}$ be a system of homogeneous coordinates for $(\bP^n)^{\ast}$. Then $D_{ij}$ is the pullback of the hypersurface
$
    \{y_iy_j^q - y_i^qy_j = 0\}\subset(\bP^n)^{\ast}
$
under the rational map
$$
    \widetilde{\gamma}\colon\bP^n\dashrightarrow (\bP^n)^{\ast}
    :[x_0:\cdots:x_n]\mapsto [F_0:\cdots:F_n]
$$
which is obtained by extending the Gauss map of $X$ to the ambient $\bP^n$. Because the intersection
$
    \cap_{i<j}\{y_iy_j^q - y_i^qy_j = 0\}
$
defines the collection of $\bF_q$-points on $(\bP^n)^\ast$, the locus $Z_X\subset X$ consists of $P\in X$ such that $T_PX$ is defined over $\bF_q$. Note that this includes the case $T_PX = \bP^n$, that is, $P$ is singular.
\end{rmk}

\begin{cor}
\label{cor:Z_X_smoothX}
Let $X\subset\bP^n$ be a smooth hypersurface over $\overline{\bF_q}$. Then $Z_X\subset X$ is a zero-dimensional subscheme which consists of $P\in X$ such that $T_PX$ is defined over $\bF_q$. Furthermore, the number of hyperplanes over $\bF_q$ tangent to $X$ is bounded by $\#Z_X(\overline{\bF_q})$.
\end{cor}

\begin{proof}
Since $X$ is smooth, its Gauss map is finite \cite{Zak93}*{Corollary~I.2.8}, which implies that $Z_X$ has dimension zero. By Proposition~\ref{prop:singularOrGroundTangent} and the smoothness, $Z_X$ consists of $P\in X$ such that $T_PX$ is defined over $\bF_q$. There exists a surjective map from $Z_X(\overline{\bF_q})$ to the set of $\bF_q$-hyperplanes tangent to $X$ which sends $P$ to $T_PX$, which proves the last assertion. 
\end{proof}

\subsection{Interplay with Frobenius classicality}
\label{subsect:ZX-and-X10}

Let $X\subset\bP^n$ be a smooth hypersurface over $\bF_q$ with defining polynomial $F = F(x_0,\dots,x_n)$. Consider the hypersurface
$$
    X_{1,0}\colonequals
    \left\{
        \sum_{i=0}^nx_i^q\frac{\partial F}{\partial x_i} = 0
    \right\}\subset \bP^n.
$$
If we let $\Phi\colon\bP^n\longrightarrow\bP^n$ denote $q$-th Frobenius endomorphism, then
$$
    (X\cap X_{1,0})(\overline{\bF_q}) = \left\{
        P\in X(\overline{\bF_q})
        \;\middle|\;
        \Phi(P)\in T_PX
    \right\}.
$$
In particular, $X$ is Frobenius classical if and only if $X_{1,0}$ does not contain $X$, that is, $X\cap X_{1, 0}$ has dimension $n-2$. On the other hand, Corollary~\ref{cor:Z_X_smoothX} asserts that $Z_X$ consists of $P\in X$ such that $\Phi(T_PX) = T_PX$, whence $Z_X\subset X\cap X_{1,0}$. We will use this relation to bound the number $\#Z_X(\overline{\bF_q})$, which will then provide a bound for the number of $\bF_q$-hyperplanes tangent to $X$.

\begin{lemma}
\label{lem:irreducible-intersect-with-Dij} 
Let $X$ be a smooth hypersurface over $\bF_q$. Suppose that $Y\subset X$ is a subscheme which is irreducible over $\bF_q$ and of dimension at least $1$. Then there exists $D_{ij}$ such that $\dim(Y\cap D_{ij}) = \dim(Y)-1$
\end{lemma}

\begin{proof}
Assume, to the contrary, that $\dim(Y\cap D_{ij}) = \dim(Y)$ for all $i<j$. Because each $D_{ij}$ is defined over $\bF_q$, the irreducibility of $Y$ over $\bF_q$ implies that $Y\subset D_{ij}$. We conclude that $Y$ is contained in $X\cap \bigcap_{i<j} D_{ij} = Z_X$, which is impossible as $\dim(Y)\geq 1$ and $\dim(Z_X) = 0$.
\end{proof}

\begin{prop}
\label{prop:intersect-with-Dij}
Let $X\subset\bP^n$ be a smooth hypersurface of degree $d$ over $\bF_q$. Suppose that $Y\subset X$ is an equidimensional subscheme over $\bF_q$ of dimension at least $1$. Then $Y\cap Z_X$ consists of at most $\deg(Y)[(d-1)(q+1)]^{\dim(Y)}$ many $\overline{\bF_q}$-points counted with multiplicity.
\end{prop}

\begin{proof}
Let us proceed by induction on $\dim(Y)$. We first establish the inductive step. Assume that the statement holds for any equidimensional subscheme of $X$ over $\bF_q$ of dimension $\ell$ for all $1\leq\ell<\dim(Y)$. If $Y$ is irreducible over $\bF_q$, then Lemma~\ref{lem:irreducible-intersect-with-Dij} shows that there exists $D_{ij}$ such that $\dim(Y\cap D_{ij})=\dim(Y)-1$. By the induction hypothesis, the intersection $(Y\cap D_{ij})\cap Z_X = Y\cap Z_X$ consists of at most
$$
    \deg(Y\cap D_{ij})[(d-1)(q+1)]^{\dim(Y)-1}
    = \deg(Y)[(d-1)(q+1)]^{\dim(Y)}
$$
many $\overline{\bF_q}$-points counted with multiplicity. If $Y$ is not irreducible over $\bF_q$, then we can express $Y = \bigcup_{s=1}^m Y_s$ where each $Y_s$ is a component irreducible over $\bF_q$ and $m\geq 2$. By applying the above result to each $Y_s$, we conclude that $Y\cap Z_X$ consists of at most
$$
    \sum_{s=1}^{m} \deg(Y_s)[(d-1)(q+1)]^{\dim(Y_s)}
    = \deg(Y)[(d-1)(q+1)]^{\dim(Y)}
$$
many $\overline{\bF_q}$-points counted with multiplicity.

The argument for the initial case $\dim(Y) = 1$ is almost the same as the inductive step. The only difference is that the induction hypothesis has to be replaced by B\'{e}zout's theorem in order to conclude that $Y\cap D_{ij}$, and thus $Y\cap Z_X$, consists of at most
$$
    \deg(Y)\deg(D_{ij}) = \deg(Y)(d-1)(q+1)
$$
many $\overline{\bF_q}$-points counted with multiplicity.
\end{proof}

\begin{cor}
\label{cor:frob-classical_ZX}
Let $X\subset\bP^n$ be a smooth Frobenius classical hypersurface of degree~$d$ over~$\bF_q$ with $n\geq 3$. Then $Z_X$ consists of at most
$$
    d(q+d-1)[(d-1)(q+1)]^{n-2}
$$
many $\overline{\bF_q}$-points counted with multiplicity. In particular, the number of $\bF_q$-hyperplanes tangent to $X$ is bounded by the number above.
\end{cor}

\begin{proof}
By applying Proposition~\ref{prop:intersect-with-Dij} to $Y=X\cap X_{1,0}$ which has dimension at least $1$ since $n\geq 3$. Since $Y\cap Z_X = (X\cap X_{1,0})\cap Z_X = Z_X$, we conclude that $Z_X$ has at most
$$
    \deg(X)\deg(X_{1,0})[(d-1)(q+1)]^{n-2}
    = d(q+d-1)[(d-1)(q+1)]^{n-2}
$$
many $\overline{\bF_q}$-points counted with multiplicity. The last assertion follows from Corollary~\ref{cor:Z_X_smoothX}.
\end{proof}

\subsection{Refinement of Ballico's result}
\label{subsect:refine_Ballico}

Let us finish the proof of Theorem~\ref{thm:refine-Ballico}.

\begin{lemma}
\label{lem:exp-poly-n}
Let $n\geq 2$, $d\geq 3$ be integers and define $c_d\colonequals(3d+3)(3d-1)^{-1}$. Then
$$
    c_d\cdot d(d-1)^{n-2}
    \geq 3d(n-2).
$$
\end{lemma}

\begin{proof}
For $d=3$, we have $c_3 = 3/2$, hence
$
    c_3 \cdot d(d-1)^{n-2}
    = (3/2)\cdot d\cdot 2^{n-2}
    \geq 3d(n-2)
$
where the last inequality uses $2^{n-2}\geq 2(n-2)$. For $d\geq 4$, we have $c_d > 1$. Using the fact that $3^{n-2}\geq 3(n-2)$, we obtain
$
    c_d\cdot d(d-1)^{n-2}
    > d\cdot 3^{n-2}
    \geq 3d(n-2).
$
\end{proof}

\begin{proof}[Proof of Theorem~\ref{thm:refine-Ballico}]
The case $d=2$ follows from Ballico's theorem (see also Proposition~\ref{prop:linear-and-quadric}), so we assume $d\geq 3$ in the remaining part of the proof. When $n=2$, the conclusion follows from \cite{Asg19Thesis}*{Theorem 3.3.1} which proved existence of a transverse $\bF_q$-line assuming $q\geq d-1$, so we may assume $n\geq 3$. By Corollary~\ref{cor:frob-classical_ZX}, there exists an $\bF_q$-hyperplane transverse to $X$ if
$$
    q^{n}+q^{n-1} + \cdots + 1
    > d(d+q-1)[(d-1)(q+1)]^{n-2}.
$$
It suffices to show that
$$
    q^{n} \geq d(d+q-1)[(d-1)(q+1)]^{n-2}.
$$
Since $q\geq c_d\cdot d(d-1)^{n-2}$, it is enough to show that
$$
    q^{n-1}\cdot c_d  \geq (d+q-1)(q+1)^{n-2},
$$
or equivalently,
\begin{equation}
\label{ineq:refined-ballico-stronger}
\left(1-\frac{1}{q+1}\right)^{n-2}\cdot c_d \geq \frac{q+d-1}{q}.
\end{equation}

Our hypothesis and Lemma~\ref{lem:exp-poly-n} imply $q\geq 3d(n-2)$, thus $\frac{1}{3d} > \frac{n-2}{q+1}$. Therefore,
\begin{equation}
\label{eqn:exp-poly-n_with-q}
    \frac{1}{1-\frac{1}{3d}}>\frac{1}{1-\frac{n-2}{q+1}}.   
\end{equation}
Using Bernoulli's inequality \cite{MP93}, which asserts that $(1+x)^{\ell} \geq 1+\ell x$ for all integer $\ell\geq 0$ and real number $x\geq -1$, we obtain:
\begin{align*}
    \left(1-\frac{1}{q+1}\right)^{n-2} c_d
    &\geq \left(1-\frac{n-2}{q+1}\right) c_d = \left(1-\frac{n-2}{q+1}\right) \frac{3d+3}{3d-1}  \\
    &= \left(1-\frac{n-2}{q+1}\right) \left(1+\frac{1}{d}\right)\left(\frac{1}{1-\frac{1}{3d}}\right) \\
    &> \left(1-\frac{n-2}{q+1}\right) \left(1+\frac{1}{d}\right)\left(\frac{1}{1-\frac{n-2}{q+1}}\right)
    \tag{by \eqref{eqn:exp-poly-n_with-q}} \\
    &=1+\frac{1}{d} \geq 1+\frac{1}{d(d-1)^{n-3}} \\ 
    &> 1+\frac{1}{q/(d-1)} = 1+\frac{d-1}{q} = \frac{q+d-1}{q}.
    \tag{since $q> d(d-1)^{n-2}$}
\end{align*}
This proves the desired inequality~\eqref{ineq:refined-ballico-stronger}, and completes the proof of the theorem.
\end{proof}

\section{Bertini theorems for reduced hypersurfaces}
\label{sect:reduced-hypersection}

In this section, we prove Theorem~\ref{main-theorem-red} by splitting the task into two parts. The first result handles the case $n\geq 3$.

\begin{thm}
\label{thm:reduced-hyperplane}
Let $X\subset \bP^n$ be a reduced hypersurface of degree $d\geq 2$ over $\bF_q$ where $n\geq 3$. Then there exists a hyperplane $H\subset\bP^n$ over $\bF_q$ such that $X\cap H$ is reduced and has dimension $n-2$ provided that
\begin{itemize}
    \item $q\geq d(d-1)+1$ when $n=3$,
    \item $q\geq d$ when $n\geq 4$.
\end{itemize}
\end{thm}

The next result handles the case $n=2$.

\begin{thm}
\label{thm:tranLine_reducedPlaneCurve}
Let $C\subset \bP^2$ be a reduced curve of degree $d\geq 2$ over $\bF_q$. Suppose that
$$
    q\geq\frac{3}{2}d(d-1).
$$
Then there exists an $\bF_q$-line $L\subset\bP^2$ which is transverse to $C$.
\end{thm}

\subsection{Existence of reduced hyperplane sections}
\label{subsect:existence_reducedHyper}

The next lemma uses the scheme $Z_X$ defined previously in Section~\ref{subsect:ground_tangent}. Since $X$ is not necessarily smooth, it is possible that $\dim(Z_X)\geq 1$.

\begin{lemma}
\label{lemma:not_in_D}
Let $X\subset\bP^n$ be a reduced hypersurface and $X'\subset X_{\overline{\bF_q}}$ be an irreducible component of degree at least $2$. Then $X'\not\subset Z_X$. In particular, there exists $D_{ij}$ such that $X'\cap D_{ij}$ has dimension $n-2$.
\end{lemma}

\begin{proof}
Assume, to the contrary, that $X'\subset Z_X$. Then, for each $P\in X'$, Proposition~\ref{prop:singularOrGroundTangent} implies that either $P$ is a singular point or $T_PX$ is defined over $\bF_q$. Therefore, $X'$ is contained in the union of the singular locus $\Sing(X)$ and all the hyperplanes over $\bF_q$. It follows that $X'\subset\Sing(X)$ because $X'$ is irreducible of degree at least $2$. However, $X$ being reduced implies that $\Sing(X)$ has dimension at most $n-2$, which leads to a contradiction as $\dim(X') = n-1$. We conclude that $X'\not\subset Z_X$, and the last statement follows as $X'$ is irreducible.
\end{proof}

Let $X\subset\bP^n$ be a reduced hypersurface of degree $d$ over $\bF_q$. Write 
$
    X_{\overline{\bF_q}} = \cup_{i=1}^\ell X_i
$
where each $X_i$ is irreducible and set $d_i\colonequals\deg(X_i)$. After rearranging the indices, we may assume there exists $0\leq t\leq \ell$ such that $d_i = 1$ for $1\leq i\leq t$ and $d_i > 1$ otherwise. To find a reduced hyperplane section on $X$ over $\bF_q$, we need to estimate the number of \emph{bad} hyperplanes, namely, the hyperplanes $H$ over $\bF_q$ which is a component of $X$, or intersects $X$ properly with the intersection $X\cap H$ being non-reduced. For such an $H$, let us consider the following three mutually exclusive conditions:
\begin{enumerate}[label=(\Roman*)]
    \item\label{dim=n-1}
    $H$ is a linear component of $X$, so that $H = X_i$ for some $1\leq i\leq t$.
\end{enumerate}

In the next two conditions, we assume that $H$ intersects $X$ properly such that $X\cap H$ is non-reduced. If $X$ is defined by the polynomial $F$, the properness implies that $X\cap H$ is a hypersurface in $H\cong\bP^{n-1}$ defined by $F|_H = 0$. Consider the factorization of $F|_H$ over the algebraic closure:
\begin{equation}
\label{eqn:factor_FH}
    F|_H = \prod_{j=1}^mG_j^{m_j}
\end{equation}
where the factors $G_j$ are irreducible over $\overline{\bF_q}$ and coprime to each other. As explained before Theorem~\ref{main-theorem-red} in Section~\ref{sect:intro}, the assumption that $X\cap H$ is non-reduced means that $m_j\geq 2$ for some $1\leq j\leq m$.
\begin{enumerate}[label=(\Roman*)]
\setcounter{enumi}{1}
    \item\label{proper_linear}
    $H$ intersects $X$ properly and passes through the intersection of two distinct linear components $X_i\cap X_j$ where $1\leq i < j\leq t$. In this case, $(X\cap H)_{\overline{\bF_q}}$ has $X_i\cap X_j$ as an irreducible component along which $(X\cap H)_{\overline{\bF_q}}$ is non-reduced. More explicitly, this means that there exists a factor $G_\alpha$ with $\deg(G_\alpha) = 1$ and $m_\alpha\geq 2$ in \eqref{eqn:factor_FH} such that $\{G_\alpha = 0\} = X_i\cap X_j$ as hyperplanes in $H\cong\bP^{n-1}$.
    \item\label{proper_nonlinear}
    $H$ intersects $X$ properly and does not pass through the intersection of any two linear components of $X$. Hence, if $Y$ is a component of $(X\cap H)_{\overline{\bF_q}}$ along which $(X\cap H)_{\overline{\bF_q}}$ is non-reduced, then $Y\subset X_i$ for some $i > t$. Note that $Y = \{G_\alpha = 0\}$ where $G_\alpha$ is a factor in \eqref{eqn:factor_FH} with $m_\alpha\geq 2$.
\end{enumerate}

\begin{lemma}
\label{lemma:bound_I-and-II}
The number of hyperplanes of type~\ref{dim=n-1} or \ref{proper_linear} is bounded by
$$
    \binom{t}{2}\cdot(q+1) + 1
    = \frac{1}{2}t(t+1)(q+1) + 1.
$$
\end{lemma}

\begin{proof}
Suppose that $t\leq 1$, that is, $X$ contains at most one linear component. Then there is at most one hyperplane of type~\ref{dim=n-1} and no hyperplane of type~\ref{proper_linear}. Hence the number of hyperplanes under consideration is bounded by $1$.

Suppose that $X$ contains at least two linear components. Then a hyperplane $H$ of type~\ref{dim=n-1} or \ref{proper_linear} must contain $X_i\cap X_j$ for some $1\leq i<j\leq t$. Note that there are at most $q+1$ many $\bF_q$-hyperplanes pass through $X_i\cap X_j$ due to, for example, the fact that the intersection of more than $q+1$ many $\bF_q$-hyperplanes in $\bP^n$ has codimension at least $3$. Therefore, the bound in this case is given by 
$
    \binom{t}{2}\cdot(q+1).
$

Adding the two bounds in the above two cases together gives the desired bound.
\end{proof}

Given a hyperplane $H$ of type~\ref{proper_nonlinear}, we define $\mathcal{A}_H$ to be the collection of subschemes $Y\subset X$ which can be written as $Y = \{G_\alpha = 0\}$ for some $G_\alpha$ with $m_\alpha\geq 2$ in \eqref{eqn:factor_FH}. Given any $Y\in\mathcal{A}_H$ and any $P\in Y(\overline{\bF_q})$, either $T_PX = H$ or $P$ is a singular point of $X$. It then follows from Proposition~\ref{prop:singularOrGroundTangent} that
\begin{equation}
\label{eqn:Y_in_Z}
    Y\subset Z_X = X\cap\bigcap_{0\leq i<j\leq n}D_{ij}.
\end{equation}
Let us further take the union
$$
    \mathcal{B}\colonequals
    \bigcup_{
        H\text{ satisfies \ref{proper_nonlinear}}
    }
    \mathcal{A}_H.
$$

\begin{lemma}
\label{lemma:cardinality_B}
The cardinality of $\mathcal{B}$ is bounded by
$
   (d-t)(d-1)(q+1).
$
Therefore, the number of hyperplanes of type~\ref{proper_nonlinear} is bounded by
$
   (d-t)(d-1)(q+1)^2.
$
\end{lemma}

\begin{proof}
Let $Y\in\mathcal{B}$. Then $Y\subset X_k$ for some $k\geq t+1$. Because $\deg(X_k)\geq 2$, there exists $D_{ij}$ such that $\dim(X_k\cap D_{ij}) = n-2$ by Lemma~\ref{lemma:not_in_D}. Hence $Y$ is the underlying reduced subscheme of an irreducible component of $(X_k\cap D_{ij})_{\overline{\bF_q}}$ due to \eqref{eqn:Y_in_Z}. By B\'ezout's theorem, the number of irreducible components of $(X_k\cap D_{ij})_{\overline{\bF_q}}$ is bounded by
$$
    \deg(X_k\cap D_{ij})
    = \deg(X_k)\deg(D_{ij}) 
    = d_k (d-1) (q+1).
$$
Hence the cardinality of $\mathcal{B}$ is bounded by
$$
   \sum_{k=t+1}^{\ell} d_k (d-1)(q+1)
   = (d-t)(d-1)(q+1).
$$
This proves the first statement. The second statement holds since there are at most $q+1$ many $\bF_q$-hyperplanes passing through each $Y\in\mathcal{B}$.
\end{proof}

\begin{prop}
\label{prop:nonreduced-hyperplane}
Let $X\subset \bP^n$ be a reduced hypersurface of degree $d$ over $\bF_q$. Then the number of $\bF_q$-hyperplanes $H\subset\bP^n$ such that $H$ does not intersect $X$ properly or $X\cap H$ is non-reduced is bounded by the number
$$
    (d-t)(d-1)(q+1)^2
    + \frac{1}{2}t(t-1)(q+1)
    + 1
    \;\leq\; d(d-1)(q+1)^2 + 1.
$$
\end{prop}

\begin{proof}
A hyperplane $H$ as in the statement is of type~\ref{dim=n-1}, \ref{proper_linear}, or \ref{proper_nonlinear}, so the bound on the left hand side is obtained by summing up the bounds in Lemmas~\ref{lemma:bound_I-and-II} and \ref{lemma:cardinality_B}. The inequality can be verified directly so we leave it to the reader.
\end{proof}

\begin{proof}[Proof of Theorem~\ref{thm:reduced-hyperplane}]
By Proposition~\ref{prop:nonreduced-hyperplane}, we will get a desirable hyperplane if 
\begin{equation}
\label{eqn:desiredIneq}
    \sum_{j=0}^n q^{j}
    > (d-t)(d-1)(q+1)^2
    + \frac{1}{2}t(t-1)(q+1)+1.
\end{equation}
We can cancel the constant $1$ on the right side by starting the sum on the left with $j=1$. In fact, we will ensure that a stronger inequality holds:
\begin{equation}
\label{eqn:goodHyperplaneExists}
    \sum_{j=1}^n q^{j}
    > \left(
        (d-t)(d-1)
        + \frac{1}{2}t(t-1)
    \right)(q+1)^2.
\end{equation}
We want to maximize the quantity
\[
    \phi(t)\colonequals
    (d-t)(d-1) + \frac{1}{2}t(t-1)
\]
as a function of $t$ on the interval $[0,d]$. Note that $\phi(t)$ is a quadratic polynomial in $t$ with the leading term $(1/2)t^2$. As the graph of $\phi(t)$ is the usual upward-facing parabola, the maximum is attained at the end point $t=0$ or $t=d$. Since $\phi(0)=d(d-1)$ and $\phi(d)=\frac{1}{2}d(d-1)$, we conclude that $\phi(t)\leq d(d-1)$.

Straightforward computations show that the inequality 
\begin{equation}
\label{eqn:q_GEQ_d}
    q^{n-3}(q-1)\geq d(d-1)
\end{equation}
holds in the following cases:
\begin{itemize}
    \item $n=3$ and $q\geq d(d-1)+1$,
    \item $n\geq 4$ and $q\geq d$.
\end{itemize}
By hypothesis, we have $n\geq 3$, which implies that 
\[
    \sum_{j=1}^n q^{j}
    > q^{n-3}(q^3+q^2-q-1)
    = q^{n-3}(q-1)(q+1)^2.
\]
Combining this with \eqref{eqn:q_GEQ_d}, we obtain
$$
    \sum_{j=1}^n q^{j}
    >d(d-1)(q+1)^2
    \geq \phi(t) (q+1)^2
    = \left(
        (d-t)(d-1) + \frac{1}{2}t(t-1)
    \right)(q+1)^2.
$$
which is exactly \eqref{eqn:goodHyperplaneExists}, as desired.
\end{proof}
\begin{rmk}
Note that inequality~\eqref{eqn:desiredIneq} fails when $n=2$ and hence necessitates a different approach in Section~\ref{subsect:reducedPlaneCurve}.
\end{rmk}

While our main Theorem~\ref{thm:reduced-hyperplane} is only stated for reduced hyperplane sections for the sake of simplicity, it easily extends by induction to the following more general result.

\begin{cor}
\label{cor:reduced-r-plane}
Let $X\subset \bP^n$ be a reduced hypersurface of degree $d\geq 2$ over $\bF_q$ with $n\geq 3$. Then, for every $2\leq r\leq n-1$, there exists an $r$-plane $T\subset\bP^n$ over $\bF_q$ such that $X\cap T$ is reduced and has the expected dimension $r-1$ provided that
$$
    q\geq d(d-1)+1.
$$
\end{cor}

The proof of the corollary is left as an easy exercise to the reader.

\subsection{Transverse lines to reduced plane curves}
\label{subsect:reducedPlaneCurve}

We prove Theorem~\ref{thm:tranLine_reducedPlaneCurve} in this section.

\begin{lemma}
\label{lemma:bertini-irreducible-curve}
Suppose that $C\subset\bP^2$ is an integral curve over $\overline{\bF_q}$ of degree $d\geq 2$. Then the number of $\bF_q$-lines not transverse to $C$ is bounded by
\[
    \frac{1}{2}(d-1)(3d-2)(q+1).
\]
\end{lemma}

\begin{proof}
An $\bF_q$-line $L$ is not transverse to $C$ either if it passes through a singular point of $C$ or if $L = T_PC$ for some $P\in C$. Since $C$ is irreducible, one can derive from, for example, \cite{Liu02}*{\S7.5, Proposition~5.4}, that the number of singular points of $C$ is at most
\[
    \frac{1}{2}(d-1)(d-2).
\]
Because each singular point has at most $q+1$ distinct $\bF_q$-lines passing through it, this accounts for
\[
    \frac{1}{2}(d-1)(d-2)(q+1)
\]
many non-transverse $\bF_q$-lines.

To estimate the number of the second type of non-transverse lines, first note that the condition $L = T_PC$ implies that $T_PC$ is defined over $\bF_q$, thus $P\in Z_C$ by Proposition~\ref{prop:singularOrGroundTangent}. As $C$ is integral, it intersects some $D_{ij}$ in $0$-dimensional scheme by Lemma~\ref{lemma:not_in_D}. Thus, the number of $\bF_q$-lines that arise as $T_PC$ is at most
\[
    C\cdot D_{ij} = \deg(C)\deg(D_{ij})=d(d-1)(q+1).
\]
Consequently, the number of non-transverse $\bF_q$-lines to $C$ is at most
$$
    \frac{1}{2}(d-1)(d-2)(q+1) + d(d-1)(q+1)
    = \frac{1}{2}(d-1)(3d-2)(q+1)
$$
as claimed.
\end{proof}

\begin{lemma}
\label{lemma:nonTransverseLines_reducedCurve}
Let $C\subset \bP^2$ be a reduced curve of degree $d\geq 2$ over $\bF_q$. Then the number of $\bF_q$-lines not transverse to $C$ is bounded by
$$
    \frac{3}{2}d(d-1)(q+1).
$$
\end{lemma}

\begin{proof}
Write
$
    C_{\overline{\bF_q}} = \cup_{i=1}^\ell C_i
$
where each $C_i$ is irreducible and let $d_i\colonequals\deg(C_i)$.
For each $\bF_q$-line $L$ not transverse to $C$, we have that
\begin{enumerate}[label=(\roman*)]
    \item\label{non-transverse_component}
    $L$ meets $C_i$ non-transversely for some $i$ where $\deg(C_i)\geq 2$, or
    \item\label{non-transverse_intersection}
    $L$ passes through an intersection point of $C_i$ and $C_j$ for some $i\neq j$. Note that this includes $L$ which meets $C_i$ non-transversely where $\deg(C_i) = 1$.
\end{enumerate}
By applying Lemma~\ref{lemma:bertini-irreducible-curve} to each component $C_i$ and summing up all the upper bounds, we conclude that the number of lines in \ref{non-transverse_component} is at most
$$
    \frac{1}{2}\sum_{i=1}^{\ell}(d_i-1)(3d_i-2)(q+1).
$$
On the other hand, the number of points in $C_i\cap C_j$ for $i\neq j$ is at most $d_id_j$ by B\'ezout's theorem. Since there are at most $(q+1)$ lines defined over $\bF_q$ that pass through a point in $C_i\cap C_j$, the number of lines in \ref{non-transverse_intersection} is at most
$$
    \sum_{i<j} d_id_j(q+1).
$$

By adding up all the contributions above, we obtain that the number of $\bF_q$-lines not transverse to $C$ is at most
\[
    (q+1)\left(\frac{1}{2}\sum_{i=1}^{\ell}(d_i-1)(3d_i-2)
    + \sum_{i<j} d_id_j\right)
\]
As a factor of the above, we have 
\begin{align*}
    \frac{1}{2}\sum_{i=1}^{\ell}(d_i-1)(3d_i-2)
    + \sum_{i<j} d_id_j
    &= \sum_{i=1}^{\ell} \left(\frac{3}{2} d_i^2 - \frac{5}{2} d_i + 1\right) + \sum_{i<j} d_i d_j \\
    &= \sum_{i=1}^{\ell} \left(d_i^2 - \frac{5}{2} d_i + 1\right) + \frac{1}{2}\left(\sum_{i=1}^{\ell} d_i^2 + 2\sum_{i<j} d_i d_j\right) \\
    &= \sum_{i=1}^{\ell} \left(d_i^2 - \frac{5}{2} d_i + 1\right) + \frac{1}{2}\left(\sum_{i=1}^{\ell} d_i\right)^2 \\
    &= \left(\sum_{i=1}^{\ell} d_i^2  \right) - \frac{5}{2}d + \ell+ \frac{1}{2}d^2 
\end{align*}
where the last equality uses $d=\sum_{i=1}^{\ell} d_i$.
Using the facts that
$$
    \sum_{i=1}^{\ell}d_i^2 \leq \left(\sum_{i=1}^{\ell}d_i\right)^2 = d^2
    \quad\text{and}\quad
    \ell\leq d,
$$
we conclude that the number of $\bF_q$-lines not transverse to $C$ is at most
$$
    \left(d^2-\frac{5}{2}d+d+\frac{1}{2}d^2\right)(q+1)
    = \left(\frac{3}{2}d^2-\frac{3}{2}d\right)(q+1)
    = \frac{3}{2}d(d-1)(q+1)
$$
as desired.
\end{proof}

\begin{proof}[Proof of Theorem~ \ref{thm:tranLine_reducedPlaneCurve}]
The number of $\bF_q$-lines in $\bP^2$ is $q^2+q+1$, so, by Lemma~\ref{lemma:nonTransverseLines_reducedCurve}, there exists a transverse $\bF_q$-line if
$$
    q^2+q+1
    >\frac{3}{2}d(d-1)(q+1).
$$
This inequality holds under the hypothesis $q\geq \frac{3}{2}d(d-1)$. Indeed, we have
$$
    q^2+q+1 > q^2+q=q(q+1) \geq \frac{3}{2}d(d-1)(q+1)
$$
which completes the proof.
\end{proof}

\begin{rmk}
There is a different method \cite{AG}*{Proposition 2.2} to prove Theorem~\ref{thm:tranLine_reducedPlaneCurve} at the cost of slightly stronger hypothesis $q\geq 2d(d-1)$.
\end{rmk}

Our final result in the present paper concerns existence of transverse $\bF_q$-lines to reduced hypersurfaces of arbitrary dimension. We obtain it by reducing the statement to the case of plane curves (Theorem~\ref{thm:tranLine_reducedPlaneCurve}) with the help of Corollary~\ref{cor:reduced-r-plane}.

\begin{cor}
\label{cor:tranLine_reducedHypersurface}
Let $X\subset \bP^n$ be a reduced hypersurface of degree $d$ defined over $\bF_q$. Suppose that
\[
    q\geq\frac{3}{2}d(d-1).
\]
Then there exists an $\bF_q$-line $L\subset\bP^n$ which is transverse to $X$.
\end{cor}

\begin{proof}
When $d=1$, $X$ is a hyperplane and any $\mathbb{F}_q$-line $L\subset \mathbb{P}^n$ with $L\not \subset H$ satisfies the conclusion. For $d\geq 2$, it is straightforward to see that the inequality $q\geq \frac{3}{2}d(d-1)$ implies $q\geq d(d-1)+1$. By Corollary~\ref{cor:reduced-r-plane}, there exists a plane $H\cong\bP^2$ in $\bP^n$ over $\bF_q$ such that $X_{1}\colonequals X\cap H$ is a reduced plane curve. Now we apply Theorem~\ref{thm:tranLine_reducedPlaneCurve} to find an $\bF_q$-line  $L\subset\bP^2$ such that $X_{1}\cap L$ consists of $d$ distinct points. This line $L$ also satisfies the condition that $\#(X\cap L)=d$ distinct points, and  so $L$ is a desired transverse line to $X$. 
\end{proof}

\bibliographystyle{alpha}
\bibliography{Transversality}

\ContactInfo

\end{document}